\documentclass[11pt,reqno]{amsart}
\usepackage{latexsym}
\RequirePackage{amsmath}   \RequirePackage{amssymb}
\RequirePackage{amsthm}   \RequirePackage{epsfig}
\theoremstyle{plain}
\newtheorem{theorem}{Theorem}
\newtheorem{proposition}[theorem]{Proposition}
\newtheorem{corollary}[theorem]{Corollary}
\newtheorem{lemma}[theorem]{Lemma}
\newtheorem*{corollary*}{Corollary}

\newtheorem*{conjecture*}{Conjecture}

\theoremstyle{definition}

\theoremstyle{remark}

\newcommand {\SR}{{\mathbb R}}

    \newcommand{\la}{\lambda}
  \newcommand{\vp}{\varphi}

\newcommand{\be}{\begin{equation}}
\newcommand{\ee}{\end{equation}}
\newcommand{\bea}{\begin{eqnarray}}
\newcommand{\eea}{\end{eqnarray}}

\usepackage{xcolor}

\title[Some more counterexamples for Bombieri's conjecture]{Some more counterexamples for Bombieri's conjecture on univalent functions}
\author{Iason Efraimidis \and Carlos Pastor}
\date{}


\address{Facultad de Matem\'aticas, Pontificia Universidad Cat\'olica de Chile, Avenida Vicu\~na Mackenna 4860, Santiago, Chile.} \email{iason.efraimidis@mat.uc.cl}
\address{Instituto de Ciencias Matem\'aticas, Nicol\'as Cabrera 13-15, 28049 Madrid, Spain.} \email{carlos.pastor@icmat.es}

\begin{document}
%%%%%%%%%%%%%%%%%%%%%%%%%%%
\maketitle

\begin{abstract}
We disprove a conjecture of Bombieri regarding univalent functions in the unit disk in some previously unknown cases. The key step in the argument is showing that the global minimum of the real function $\big(n\sin{x}-\sin(nx)\big)/\big(m\sin{x}-\sin(mx)\big)$ is attained at $x = 0$ for integers $m>n\geq2$ when $m$ is odd and $n$ is even, $m$ is sufficiently big and $0.5 \leq n/m \leq 0.8194$.
\end{abstract}

\section{Introduction}
%%%%%%%%%%%%%%%%%%%%%%%%%%%

Let $S$ denote the class of analytic functions
$$
f(z) = z + a_2z^2 + a_3z^3 + \ldots + a_n z^n + \ldots \qquad \big(|z| < 1\big)
$$
which are univalent in the unit disk. A relevant member of this class is the Koebe function $K(z) = z/(1-z)^2$. Bombieri conjectured in \cite{bombieri} that one should have 
\begin{equation}\label{bombieri}
\sigma_{mn}  \, = \, B_{mn} \qquad \big(m, n \geq 2\big), 
\end{equation}
where
$$
\sigma_{mn} \, = \, \liminf_{f \rightarrow K} \frac{n - {\rm Re}\, a_n}{m - {\rm Re}\, a_m} \qquad \text{and} \qquad B_{mn} \, = \, \min_{x \in \mathbb{R}} \frac{n\sin{x} - \sin(nx)}{m\sin{x} - \sin(mx)}
$$
and the limit in $\sigma_{mn}$ is taken inside the class $S$ in the sense of uniform convergence over compact sets. It is known that 
$$
0  \, \leq \,  \sigma_{mn}  \, \leq \, B_{mn} \qquad \big(m, n \geq 2\big), 
$$
where the first inequality is a consequence of the local maximum property of the Koebe function while the second is a theorem of Prokhorov and Roth \cite{prokhorov_roth}. It is easy to see that $B_{mn}=0$ when $m$ is even and $n$ is odd and, therefore, that the conjecture is true in this case since $\sigma_{mn}  = B_{mn}=0$. This conjecture has also been verified by Bshouty and Hengartner \cite{BH} for analytic variations of the Koebe function and for functions with real coefficients (a simpler proof of the latter was given in \cite{prokhorov_roth}). However, the number of known counterexamples in the remaining cases has been steadily increasing. 

Recently Leung \cite{leung} devised a variational method (using Loewner's theory) to construct a uniparametric family of functions $f_\varepsilon$ in the class $S$ converging to the Koebe function as $\varepsilon \rightarrow 0^+$ and satisfying
$$
 \sigma_{mn}  \, \leq \, \lim_{\varepsilon \rightarrow 0^+} \frac{n - {\rm Re}\, a_n}{m - {\rm Re}\, a_m} \, < \, \frac{n^3-n}{m^3-m} \qquad \big(m>n\geq 2\big).
$$
This, therefore, yields a counterexample to \eqref{bombieri} as long as the function
$$
f(x) = \frac{n\sin{x}-\sin(nx)}{m\sin{x}-\sin(mx)}, \qquad \text{for which} \quad f(0)=\frac{n^3-n}{m^3-m},  
$$
satisfies
\begin{equation}\label{works}
\min_{x\in\mathbb{R}} f(x) = f(0).
\end{equation}
Condition \eqref{works} was verified by the first author in \cite{efraimidis} for $m > n \geq 2$ when $m$ and $n$ have the same parity and in the case when $m$ is odd, $n$ is even and $n \leq (m+1)/2$, thus disproving Bombieri's conjecture for all these pairs of integers $(m, n)$.  

In this article we prove the following theorem. 
\begin{theorem} \label{thm}
Let $m$ and $n$ be integers such that $m$ is odd, $n$ is even, $0.5 \leq n/m \leq 0.8194$ and $m \geq 81$. Then condition \eqref{works} is satisfied. In particular, Bombieri's conjecture \eqref{bombieri} fails for all these pairs of integers. 
\end{theorem}

In \cite{efraimidis} it was conjectured that \eqref{works} should hold for all $m > n \geq 2$ such that $m$ is odd, $n$ is even and $n < (4m+2)/5$. Note that Theorem~\ref{thm} not only proves this conjecture for $m \geq 81$, leaving only finitely many cases open, but also shows that the slope $4/5$ is not optimal. Determining the critical slope
$$
\inf\big\{n/m \; : \; m > n \geq 2,\; m \text{ odd},\; n \text{ even},\; \eqref{works} \text{ fails}\big\}
$$
is still an open problem.

To prove Theorem \ref{thm} we employ simple real variable techniques, often involving Taylor series. Computer-assisted graphs\footnote{These graphs were made using the online applet www.desmos.com/calculator.} suggest that for $n/m \geq 0.82$ condition \eqref{works} is no longer satisfied. A variation of the arguments presented in this article should suffice to show that this is indeed the case.

For the rest of the article we shall assume that $m> n \geq 2$ are integers, $m$ being odd and $n$ even, and that $\lambda = n/m$ lies in the interval $[0.5, 0.82]$. Instead of studying the function $f(x)$ directly, we will consider
$$
g(x) \, = \, mf(x) - n \, = \, \frac{n\sin(mx) - m\sin(nx)}{m\sin{x} - \sin(mx)}.
$$
This function has two advantages over the former $f$: firstly it oscillates around $0$, and secondly its value at the origin
$$
g(0) = -\frac{n(m^2-n^2)}{m^2-1}
$$
behaves asymptotically like $-\lambda(1-\lambda^2)m$, thus diverging. This makes the coefficients involved in the expression defining $g$ comparatively smaller than in the case of $f$ where $f(0)$ tends to $\la^3$, allowing for better approximations.

The proof of \eqref{works} is divided in three parts: in section \ref{s_far} we show that $g$ cannot attain any value below $g(0)$ except, possibly, in a neighborhood of either $0$ or $\pi$. These neighborhoods are then studied more closely in sections~\ref{s_0} and \ref{s_pi} to show that the minimum is indeed attained at $0$. An appendix is devoted to some tedious calculations.

\section{Away from $0$ and $\pi$}\label{s_far}  
%%%%%%%%%%%%%%%%%%%%%%%%%%%

One can readily show that $g$ cannot attain its global minimum away from $0$ and $\pi$ by exploiting the simple estimate
$$
|g(x)| \, \leq \, \frac{m+n}{\big|m|\sin{x}|-1\big|}.
$$
We will prove the following sharper version of the lower bound given by this inequality:  

\begin{proposition}\label{prop_far}
The function $g$ satisfies
$$
g(x) \, \geq \, -\frac{m+n}{m|\sin{x}|+1} \qquad \big(x \in \mathbb{R}\big).
$$
\end{proposition}

As an immediate consequence we deduce that $g$ cannot attain its global minimum in the region determined by $-(m+n)/(m|\sin{x}|+1) > g(0)$, or equivalently
\begin{equation}\label{far}
|\sin{x}| \, > \, \frac{m^2 - mn + n^2 - 1}{mn(m-n)}.
\end{equation} 
This region can be further simplified to obtain the following weaker but more explicit result:

\begin{corollary}
The function $g$ satisfies $g(x) \geq g(0)$ in the interval $[5.78/m, \pi - 5.78/m]$ for all $\, m \geq 81$.
\end{corollary}

\begin{proof}
By the symmetry of \eqref{far} it suffices to prove that the inequality is satisfied for $x = 5.78/m$. Furthermore, since $\sin{x} \geq x - x^3/6$ for $x > 0$, it will be enough to show that 
$$
\frac{5.78}{m} - \frac{5.78^3}{6m^3} \, \geq \, \frac{m^2 - mn + n^2}{nm(m-n)} = \frac{1-\lambda+\lambda^2}{m\lambda(1-\lambda)}.
$$
Recall that $0.5\leq \lambda \leq 0.82$ and note that the function on the right-hand side is increasing in $\lambda$. Hence, it suffices to prove that this inequality holds for $\lambda = 0.82$. An easy computation shows that this is true when $m \geq 81$. 
\end{proof}

\begin{proof}[Proof of Proposition~\ref{prop_far}]
We may restrict ourselves to $x \in [0, \pi]$. Note that $g(0)\geq - (m+n)$ and that $g(\pi)=+\infty$. For $x\in (0,\pi)$ we rewrite the desired inequality as
\begin{equation}\label{lower2}
m\sin(nx)-n\sin(mx) + \frac{\sin(nx)}{\sin{x}} + \frac{\sin(mx)}{\sin{x}} \, \leq \, m+n.
\end{equation}

We follow different arguments depending on whether $\sin{x} \leq 1/n$ or $1/n \leq \sin{x} \leq 1$. In the first case we rewrite the left-hand side of \eqref{lower2} as
$$
 m +n - \left(m+\frac{1}{\sin{x}}\right) \big(n\sin{x} - \sin(nx)\big) - \left(\frac{1}{\sin{x}}-n\right)\big(m\sin{x}-\sin(mx)\big).
$$
The inequality now follows by noting that $n \sin{x} - \sin(nx) \geq 0$ and $m \sin{x} - \sin(mx) \geq 0$. When $\sin{x} \geq 1/n$ we rewrite the left-hand side of \eqref{lower2} as
$$
\left(m+\frac{1}{\sin{x}}\right) \sin(nx) - \left(n - \frac{1}{\sin{x}}\right)\sin(mx),
$$
and note that its value increases if we replace $\sin(nx)$ by $1$ and $\sin(mx)$ by $-1$.
\end{proof}

\section{In a neighborhood of $0$}\label{s_0}
%%%%%%%%%%%%%%%%%%%%%%%%%%%

In this section we approximate the function $g$ in a neighborhood of $0$ with the objective of proving:
\begin{proposition}\label{prop_0}
The function $g$ satisfies $g(x) \geq g(0)$ in the interval $[0, 5.78/m]$ for all $\, m \geq 81$.
\end{proposition}

Our strategy will be the following: let $s_1$ and $s_2$ be two functions satisfying 
\begin{equation}\label{cond1}
s_1(t) \leq \sin{t} \leq s_2(t) \qquad \big(t \in [0, 5.78]\big)
\end{equation}
and
\begin{equation}\label{cond2}
ms_1(x) - s_2(mx) \geq 0 \qquad \big(x \in [0, 5.78/m], \; m \geq 81\big).
\end{equation}
Under these circumstances, for $x \in [0, 5.78/m]$ and $m \geq 81$, we have that either $g(x) \geq 0$ or 
$$g(x) \geq \frac{ns_1(mx)-ms_2(nx)}{ms_1(x)-s_2(mx)}.$$
Proposition~\ref{prop_0} will then follow if for an appropriate choice of $s_1$ and $s_2$ we can show that
\begin{equation}\label{cond3}
\frac{ns_1(mx)-ms_2(nx)}{ms_1(x)-s_2(mx)} \geq g(0) \qquad \big(x \in [0, 5.78/m], \; m \geq 81\big).
\end{equation}

\begin{proof}[Proof of Proposition~\ref{prop_0}]
We consider
\begin{align*}
s_1(x) &= x-\frac{x^3}{3!}+\frac{x^5}{5!}-\frac{x^7}{7!}+ \frac{x^9}{482800} \\
s_2(x) &= x-\frac{x^3}{3!}+\frac{x^5}{5!}-\frac{x^7}{7!}+ \frac{x^9}{9!}
\end{align*}
for which we will now check that conditions (\ref{cond1}-\ref{cond3}) hold. We begin with condition \eqref{cond1} and note that integrating the inequality $\cos{x} \leq 1$ nine times from $0$ to $x$, we readily obtain $\sin{x} \leq s_2(x)$ for $x \geq 0$. 

%%%%%%%%%%%%%%%%
For the other inequality in \eqref{cond1} we will prove that the first eight derivatives of $\vp(x) = \sin{x} - s_1(x)$, including $\vp$ itself, are positive on some initial interval and vanish at most at a single point for $x > 0$. Note that this claim reduces the second inequality in \eqref{cond1} to checking that $\vp(5.78)\approx 0.0104 >0$.

We show first that the claim holds for the eighth derivative. To do this we write 
$$
\psi(x) = \vp^{(8)}(x) = \sin x - ax, \qquad \text{where} \quad  a= \frac{9!}{482800} \approx  0.7516,  
$$
and note that $\psi(0) =0$, $\psi'(0) >0$ and $\psi(x)<0$ for $x>1/a\approx 1.3305$. Now for $x>0$ the first two roots of $\psi'$ are $x_1= \arccos(a)\approx 0.7203$ and $x_2=2\pi-\arccos(a)\approx 5.5629$. Therefore $\psi$ is positive in $(0,x_1]$, decreases and has a single root in $(x_1,x_2)$ and is strictly negative in $[x_2,+\infty)$. The claim for the remaining derivatives is now proved by backwards induction, noting that $\vp^{(k)}(0) = 0$ for $0 \leq k \leq 7$ and that the claim for $\vp^{(k+1)}$ implies $\vp^{(k)}$ is increasing in a neighborhood of zero and has at most one critical point.

%%%%%%%%%%%%%%%%
To verify \eqref{cond2} we rewrite it, after dividing by $m^3x^3$, as
$$
\frac{1}{3!}-\frac{m^2x^2}{5!}+\frac{m^4x^4}{7!}-\frac{m^6x^6}{9!} - \frac{1}{3!m^2} + \frac{x^2}{5!m^2} - \frac{x^4}{7!m^2} + \frac{x^6}{482800m^2} \geq 0. 
$$
The change of variables $y=m^2x^2\in [0, 5.78^2]$ shows that the above is equivalent to $p(y)-q(y) \geq 0$, where
$$
p(y) = \frac{1}{3!}-\frac{y}{5!}+\frac{y^2}{7!}-\frac{y^3}{9!} \qquad \text{and} \qquad q(y) = \frac{1}{3!m^2} - \frac{y}{5!m^4} + \frac{y^2}{7!m^6} - \frac{y^3}{482800m^8}. 
$$
We compute 
$$
p'(y) \, = \, -\frac{1}{7!} \left(  \frac{y^2}{24} -2y +42\right)  \, < \, 0 \qquad (y \in\SR)
$$
and, since $p$ decreases, we deduce that 
$$
p(y) \,\geq\, p(5.78^2) \,\approx \, 6.9607\cdot 10^{-3}\qquad \big(y \in [0, 5.78^2]\big). 
$$
Also, it is easy to check that
$$
q'(y) \, = \, -\frac{1}{7!m^6} \left(  \frac{3\cdot 7!}{482800 m^2} \, y^2 -2y +42m^2\right)  \, < \, 0 \qquad (y \in\SR), 
$$
which shows that $q$ decreases and therefore $q(y)\leq q(0)$ for $y\geq 0$. Hence
$$
p(y)-q(y) \, > \, 6.96 \cdot 10^{-3} - \frac{1}{3! m^2} \qquad \big(y \in [0, 5.78^2]\big), 
$$
which can easily be seen to be positive for $m\geq 5$. This proves condition \eqref{cond2}. 

Finally, in order to verify condition \eqref{cond3}, we multiply by its denominator, divide by $mnx^3$ and rewrite it as 
\begin{align*}
\frac{n^2}{3!} -\frac{n^4x^2}{5!} +\frac{n^6x^4}{7!} -\frac{n^8x^6}{9!} -\frac{m^2}{3!} +\frac{m^4 x^2}{5!} -\frac{m^6 x^4}{7!} +\frac{m^8 x^6}{482800} \geq -\frac{m^2(m^2-n^2)}{m^2-1}\times &\\
\times \left( \frac{1}{3!} -\frac{m^2x^2}{5!}+\frac{m^4x^4}{7!}-\frac{m^6x^6}{9!}  -\frac{1}{3!m^2} + \frac{x^2}{5!m^2} - \frac{x^4}{7!m^2} +\frac{x^6}{482800m^2} \right).  &
\end{align*}
We cancel out the constant terms, divide by $(m^2-n^2)x^2$ and regroup to get
\begin{align*}
\left[ \frac{1}{482800} \left( \frac{m^8}{m^2-n^2} +\frac{1}{m^2-1} \right) - \frac{1}{9!} \left( \frac{n^8}{m^2-n^2} +\frac{m^8}{m^2-1} \right)\right] x^4 &\\
 +\frac{m^2+1-m^2n^2-n^4}{7!} \,x^2 + \frac{n^2-1}{5!} & \, \geq \, 0.  
\end{align*}
We rescale with $y=m^2x^2 \in [0, 5.78^2]$, substitute $n=\la m$ and divide by $m^2$ to see that the above is equivalent to
\be \label{cond3'}
\vp(y)+\psi(y)\geq 0,
\ee
where we have grouped all the terms that depend on $m$ in 
$$
\vp(y) = \left( \frac{1}{482800m^6}- \frac{1}{9!}\right)\frac{y^2}{m^2-1} +\frac{m^2+1}{7! m^4} \,y -\frac{1}{5! m^2} 
$$
and the remaining terms in
$$
\psi(y) =  \left( \frac{1}{482800}- \frac{\la^8-\la^2+1}{9!}\right)\frac{y^2}{1-\la^2} -\frac{\la^2(1+\la^2)}{7!} \,y +\frac{\la^2}{5!}. 
$$
We have
$$
\vp(y) \geq -\frac{y^2}{9!\, (m^2-1)} - \frac{1}{5!\,m^2} \geq - \frac{5.78^4}{9!\,(81^2-1)} - \frac{1}{5!\cdot 81^2} > -1.74 \cdot 10^{-6},
$$
hence it will suffice to show $\psi(y) \geq 1.74 \cdot 10^{-6}$. We first show its derivative
$$
\psi'(y) \, =  \, \frac{2}{9!\, (1-\la^2)}\Big[ (a-1+\la^2-\la^8) y -36 \la^2(1-\la^4) \Big]
$$ 
is negative (here, as before, $a= 9! /482800$). Clearly $\psi'(0) < 0$ and since $\psi'$ is linear we just need to check that $\psi'(5.78^2) \leq 0$. For this we write
$$
\psi'\big(5.78^2\big) \, =  \, \frac{2}{9!\, (1-\la^2)}\Big[ \big(5.78^2 - 36\big)\la^2 + u(\la^2) \Big],
$$
where $u(\nu) = (a-1)5.78^2+36\nu^3-5.78^2 \nu^4$. Note that $u(\nu) < 0$ as the leading coefficient is negative and it attains negative values at its two critical points $\nu = 0$ and $\nu = 27/5.78^2$. Since $5.78^2 < 36$ we conclude $\psi'(y) < 0$ for $y \in [0, 5.78^2]$.

Since $\psi$ decreases we have that $\psi(y) \geq \psi(5.78^2)$ and that our aim to prove \eqref{cond3'} has been reduced to showing that $\psi(5.78^2) \geq 1.74\cdot 10^{-6}$ or, equivalently, that 
$$
 (a-1+\nu-\nu^4) \frac{5.78^4}{9!} - \nu(1-\nu^2) \frac{5.78^2}{7!} +\nu(1-\nu) \frac{1}{5!} - 1.74\cdot 10^{-6} (1-\nu) >0, 
$$
for $\nu=\la^2\in[0.25, 0.82^2]$. We denote by $V(\nu)$ the left-hand side of this inequality and show that it is concave by computing
$$
V''(\nu) \, = \, -\frac{12\cdot 5.78^4}{9!} \nu^2 + \frac{6 \cdot 5.78^2}{7!} \nu - \frac{2}{5!},
$$
a parabola which always lies below zero. Therefore, we need only verify that the function $V$ is positive at the endpoints of our interval, which is true since 
$$
V(0.25) \approx 5.5947 \cdot 10^{-7} \qquad\text{and}\qquad V(0.82^2) \approx 6.857  \cdot 10^{-5}.
$$
The proof is now complete. 
\end{proof}

\section{In a neighborhood of $\pi$}\label{s_pi}
%%%%%%%%%%%%%%%%%%%%%%%%%%%

Here we follow the exact same strategy we employed in the last section, but this time for the function
$$
\widetilde{g}(x) = g(x + \pi) = \frac{n\sin(mx)+m\sin(nx)}{m\sin{x} - \sin(mx)}. 
$$
A single pair of functions $(s_1, s_2)$ will not suffice to cover the whole interval $[0, 5.78/m]$ in this case, and for this reason we separate the result in two different statements:

\begin{proposition}\label{prop_pi_1}
The function $\widetilde{g}$ satisfies $\widetilde{g}(x) \geq g(0)$ in the interval $[0, 2.8/m]$ for any $m \geq 3$.
\end{proposition}

\begin{proposition}\label{prop_pi_2}
The function $\widetilde{g}$ satisfies $\widetilde{g}(x) \geq g(0)$ in the interval $[2.8/m, 5.78/m]$ for any $m \geq 81$ as long as $0.5 \leq \lambda \leq 0.8194$.
\end{proposition}

\begin{proof}[Proof of Proposition~\ref{prop_pi_1}]
For this region we choose
$$
s_1(x) = x - \frac{x^3}{3!} \qquad \text{and} \qquad s_2(x) = x - \frac{x^3}{3!} + \frac{x^5}{5!}.
$$
The inequalities \eqref{cond1} follow from integrating $\cos{x} \leq 1$ three and five times, respectively, from $0$ to $x$. Inequality \eqref{cond2} for the limited range $x \in [0, 2.8/m]$, after dividing by $m^3x^3$ and setting $y=m^2x^2$, reads
$$
\frac{1}{3!} -\frac{y}{5!} -\frac{1}{3!\, m^2}\geq 0 \qquad (y\in [0, 2.8^2]),
$$ 
which is clearly satisfied for $m \geq 3$. 

Finally, in view of \eqref{cond1} and \eqref{cond2} we have that either $\widetilde{g}(x)\geq 0$ or
$$
\widetilde{g}(x) \, \geq \, \frac{ns_1(mx)+ms_1(nx)}{ms_1(x)-s_2(mx)} \qquad \big(x \in [0, 2.8/m], \, m \geq 3\big).
$$
Hence, inequality \eqref{cond3} has to be replaced by
$$
\frac{ns_1(mx)+ms_1(nx)}{ms_1(x)-s_2(mx)} \, \geq \, g(0) \qquad \big(x \in [0, 2.8/m], \, m \geq 3\big).
$$
To prove it, we multiply by the denominator, divide by $mnx$, set $y=m^2x^2$ and substitute $n=\la m$ in order to obtain
$$
2 -\frac{\lambda^2}{3} \,y - \frac{(1-\lambda^2)m^2}{5! \,(m^2-1)}  \, y^2 \, \geq \, 0 \qquad \big(y \in [0, 2.8^2], \, m \geq 3\big).
$$ 
The left-hand side is clearly increasing in $m$ and decreasing in $y$, therefore it suffices to prove the inequality for $m=3$ and $y = 2.8^2$. The resulting quadratic polynomial in $\lambda$ is also decreasing, so we only have to check that the inequality is satisfied for $\lambda = 0.82$.
\end{proof}

\begin{proof}[Proof of Proposition~\ref{prop_pi_2}]
In this case we choose
\begin{align*}
s_1(x) &= -1 + \frac{\left(x-\frac{3\pi}{2}\right)^2}{2} - \frac{\left(x-\frac{3\pi}{2}\right)^4}{4!}, \\
s_2(x) &= -1 + \frac{\left(x-\frac{3\pi}{2}\right)^2}{2}.
\end{align*}
Integrating $\sin{x} \geq -1$ from $3\pi/2$ to $x$ two and four times, \eqref{cond1} follows for any $x > 0$.

For the sake of simplicity in inequality \eqref{cond2} we replace $s_1$ by the less involved function $x - x^3/3!$ (which was proven to lie below $\sin{x}$ in the previous proof). Hence we will show that
$$
m \left( x -\frac{x^3}{3!} \right) -s_2(mx) \, \geq \, 0 \qquad \big(x \in [2.8/m, 5.78/m], \, m \geq 5 \big).
$$
Setting $t=mx \in [2.8, 5.78]$ we see that this is equivalent to 
$$
1+t-\frac{\left(t-\frac{3\pi}{2}\right)^2}{2}-\frac{t^3}{6m^2} \, \geq \,  0.  
$$
We denote the left-hand side by $v(t)$ and compute $v''(t)=-1-t/m^2<0$. Since $v$ is concave in $t$ and increases in $m$ it suffices to check this inequality at the endpoints of the prescribed interval in $t$ and for $m=5$. Indeed, the inequality is true since
$$
v(2.8) \approx 1.825 \qquad \text{and} \qquad v(5.78) \approx 4.9228. 
$$ 

In view of \eqref{cond1} and the substitute of \eqref{cond2} we have that either $\widetilde{g}(x)\geq 0$ or
$$
\widetilde{g}(x) \, \geq \, \frac{ns_1(mx)+ms_1(nx)}{m\big(x-x^3/3!\big)-s_2(mx)} \qquad \big(x \in [2.8/m, 5.78/m], m \geq 5 \big). 
$$
Hence, in order to finish the proof we will show that, instead of \eqref{cond3}, the following inequality is true
$$ 
\frac{ns_1(mx)+ms_1(nx)}{m\big(x-x^3/3!\big)-s_2(mx)}  \, \geq \, g(0) \qquad \big(x \in [2.8/m, 5.78/m], m \geq 81 \big). 
$$ 
Multiplying by the denominator, dividing by $m$, setting $t=mx$ and $n = \la m$, we see that this is equivalent to  
\begin{align}
\lambda&\left(-1+\frac{\left(t-\frac{3\pi}{2}\right)^2}{2} - \frac{\left(t-\frac{3\pi}{2}\right)^4}{4!}\right) \nonumber\\
&+ \left(-1+\frac{\left(\lambda t-\frac{3\pi}{2}\right)^2}{2} - \frac{\left(\lambda t-\frac{3\pi}{2}\right)^4}{4!}\right) \nonumber\\
&+\lambda(1-\lambda^2)\left(1+\frac{1}{m^2-1}\right)\left(t-\frac{t^3}{6m^2}+1-\frac{\left(t-\frac{3\pi}{2}\right)^2}{2}\right) \geq 0 \label{last_ineq}
\end{align}
holding uniformly for $t \in [2.8, 5.78]$, $\la \in [0.5, 0.8194]$ and $m \geq 81$. Denote by $F(\la, t, m)$ the left-hand side. We claim that this function is increasing in $m$. To simplify the computations we differentiate with respect to the variable $u = 1/m^2$, noting that $1/(m^2-1) = -1 + 1/(1-u)$, to obtain
$$
F_u(\la, t, m) = \frac{\la (1-\la^2)}{(1-u)^2}\left(t-\frac{t^3}{6}+1-\frac{\left(t- \frac{3\pi}{2} \right)^2}{2} \right).
$$
Let $p(t)$ denote the polynomial in $t$ that appears between parenthesis in the last expression. Note that $p'(t) = 1+3\pi/2 -t -t^2/2$ is decreasing in $t$ and $p'(2.8) \approx -1.0076$. Therefore $p$ is also decreasing in $t$ in the interval $[2.8, 5.78]$ and hence it suffices to check that $p(2.8) \approx -1.6873$ is indeed negative. Our claim has been proved: $F(\la, t, m)$ is increasing in $m$ and therefore \eqref{last_ineq} will follow from proving $F(\la, t, 81) \geq 0$ uniformly for $t \in [2.8, 5.78]$, $\la \in [0.5, 0.8194]$. As this set is compact, this can be done with the aid of a computer, evaluating $F$ in a sufficiently dense grid. Alternatively, we include in the appendix a (fairly tedious) proof of this fact that can be verified using a hand-held calculator.
\end{proof}

\section{Appendix}
%%%%%%%%%%%%%%%%%%%%%%%%%%%
Let $F(\la, t)$ denote the left-hand side of \eqref{last_ineq} for $m = 81$. The proof of $F(\la, t) \geq 0$ for $t \in [2.8, 5.78]$, $\la \in [0.5, 0.8194]$ is divided in two steps: first we check that we only need to verify this inequality for $\la = 0.5$ and $\la = 0.8194$, and then we deal with these special cases separately. The following two lemmas constitute the first part:

\begin{lemma}\label{lem:ll}
The inequality $F_{\la \la}(\la, t) \leq 0$ is satisfied for $2.8 \leq t \leq 5$ and $0.5 \leq \la \leq 0.82$.
\end{lemma}

\begin{lemma}\label{lem:l}
The inequality $F_{\la}(\la, t) \leq 0$ is satisfied for $5 \leq t \leq 5.78$ and $0.5 \leq \la \leq 0.82$.
\end{lemma}

Indeed, by the minimum principle, $F(\la, t) \geq \min\{ F(0.5, t), F(0.8194, t)\}$ for $2.8 \leq t \leq 5$, while $F(\la, t) \geq F(0.8194, t)$ for $5 \leq t \leq 5.78$. Now the second step is completed in view of the following two lemmas.

\begin{lemma}\label{lem:0.5}
We have $F(0.5, t) \geq 0$ for $2.8 \leq t \leq 5$.
\end{lemma}

\begin{lemma}\label{lem:0.8194}
We have $F(0.8194, t) \geq 0$ for $2.8 \leq t \leq 5.78$.
\end{lemma}

\begin{proof}[Proof of Lemma~\ref{lem:l}]
We have the identity $F_\la = f + g + h$ where
\begin{align}
f(\la, t) &= -1+\frac{\left(t-\frac{3\pi}{2}\right)^2}{2} - \frac{\left(t-\frac{3\pi}{2}\right)^4}{4!}, \nonumber\\
g(\la, t) &= t\left(\lambda t-\frac{3\pi}{2}\right) - t\frac{\left(\lambda t-\frac{3\pi}{2}\right)^3}{3!}, \nonumber\\
h(\la, t) &= (1-3\lambda^2)\left(1+\frac{1}{6560}\right)\left(t-\frac{t^3}{39366}+1-\frac{\left(t-\frac{3\pi}{2}\right)^2}{2}\right). \nonumber
\end{align}
We claim that $f(\la, t) \leq 0$. Since $|t-3\pi/2| \leq 1.1$ for $t \in [5, 5.78]$ it suffices to check that the even polynomial $p(x) = -1 + x^2/2 - x^4/4!$ attains negative values in the interval $[-1.1, 1.1]$. As $p''(x) = 1-x^2/2$ is positive in this interval, it suffices to check $p(1.1) \approx -0.456$ is negative. Our claim is proved.

We focus now our attention on $h$. The parabola $p(t) = t + 1 - (t-3\pi/2)^2/2$ attains it maximum at $t = 1+3\pi/2$, and hence $p(t) \leq 3(1+\pi)/2$. Hence for $\la \leq 1/\sqrt{3}$ we have
$$
h(\la, t) \leq (1-3\la^2) \left(1+\frac{1}{6560}\right) \cdot \frac{3(1+\pi)}{2} < 6.22(1-3\la^2).
$$
On the other hand evaluation of $p$ at the endpoints of the interval $[5, 5.78]$ shows that it is positive. In fact,
$$
p(5) \approx 5.9586 \qquad \text{and} \qquad p(5.78) \approx 6.2101,
$$
and hence $p(t)  > 5.9585$. Since $t^3/39366 < 0.0049$ we have $p(t) - t^3/39366 > 5.95$. Therefore for $\la \geq 1/\sqrt{3}$ we have
$$
h(\la, t) < (1-3\la^2) \left(1+\frac{1}{6560}\right) \cdot 5.95 < 5.95(1-3\la^2).
$$
We have therefore shown that $h(\la, t) < \varphi(\la)$ where $\varphi$ is the piecewise defined function
$$
\varphi(\la) = \begin{cases}
6.22(1-3\la^2) & \text{if } \la \leq 1/\sqrt{3} \\
5.95(1-3\la^2) & \text{if } \la \geq 1/\sqrt{3}.
\end{cases}
$$

Gathering the previous inequalities and performing the change of variables $u = \la t - 3\pi/2$ in the expression defining $g$ we have
$$
\la F_\la(\la, t) < \psi(u) + \la \varphi(\la) \qquad \text{where} \quad \psi(u) = \left(u + \frac{3\pi}{2}\right)\left(u - \frac{u^3}{3!}\right).
$$
If $\psi(u) \leq 0$ and $\la \geq 1/\sqrt{3}$ both terms are negative and there is nothing to prove. Assume therefore that either $\psi(u) > 0$ or $\la < 1/\sqrt{3}$.

Note that
$$
 -2.213 < 0.5 \cdot 5 - \frac{3\pi}{2} \leq u \leq 0.82 \cdot 5.78 - \frac{3\pi}{2} < 0.028. 
$$
We claim $\psi(u) \leq 0$ for $u \in [-2.213, 0]$, and $\psi(u) \leq \psi(0.028) \approx 0.1327$ for $u \in [0, 0.028]$. Both claims follow from the fact that $\psi''(u) = 2 - 3\pi u/2 -2u^2$ is positive in both intervals, $\psi(0) = 0$ and $\psi(-2.213) \approx -1.0165$ is negative. Hence the assumption $\psi(u) > 0$ implies $u > 0$ or, equivalently, $\la > 3\pi/(2t)$. Therefore, $\la>0.815$ if $\psi(u) > 0$. Since the function $\la (1-3\la^2)$ is decreasing in $\la$, we conclude
$$
\la F_\la(\la, t) < 0.133 + 0.815 \cdot 5.95(1-3 \cdot 0.815^2) \approx -4.6807.
$$
The case $\psi(u) > 0$ is therefore covered, and we may assume $\la < 1/\sqrt{3}$. In this case we update the upper bound on $u$ to
$$
u \leq \frac{5.78}{\sqrt{3}} - \frac{3\pi}{2} < -1.375.
$$
Since $\psi(-1.375) \approx -3.1429 < \psi(-2.213)$, the convexity of this function implies $\psi(u) < -1.016$ for $u \in [-2.213, -1.375]$, and we obtain the bound
\[
\la F_\la(\la, t) < -1.016 + 0.5 \cdot 6.22(1-3 \cdot 0.5^2) = -0.2385.\qedhere
\]
\end{proof}

\begin{proof}[Proof of Lemma~\ref{lem:ll}]
We have the identities
\begin{align*}
F_{\la \la}(\la, t) = t^2&\left(1 - \frac{\left(\lambda t-\frac{3\pi}{2}\right)^2}{2}\right) \\
&\phantom{aa}- 6 \la\left(1+\frac{1}{6560}\right)\left(t-\frac{t^3}{39366}+1-\frac{\left(t-\frac{3\pi}{2}\right)^2}{2}\right),
\end{align*}
and
$$
F_{\la \la \la}(\la, t) = - t^3\left(\lambda t-\frac{3\pi}{2}\right) -6 \left(1+\frac{1}{6560}\right)\left(t-\frac{t^3}{39366}+1-\frac{\left(t-\frac{3\pi}{2}\right)^2}{2}\right).
$$
We claim $F_{\la \la \la}(\la, t) \geq 0$ for $t \in [2.8, 5]$ and $\la \in [0.5, 0.82]$. Indeed,
$$
F_{\la \la \la}(\la, t) > - t^3\left(0.82t-\frac{3\pi}{2}\right) - 6 \left(1+\frac{1}{6560}\right) \left(5+1\right)
$$
and setting $p(t)$ to be the right-hand side of this inequality we find that $p'(t) = t^2(9\pi/2 - 3.28t)$ has a single zero $a = 9\pi/6.56 \approx 4.3101$ in the interval $[2.8, 5]$ where it changes sign and therefore the function $p$ must be increasing in $[2.8, a]$ and decreasing in $[a, 5]$. Hence, it suffices to check that both $p(2.8)$ and $p(5)$ are positive. Indeed
$$
p(2.8) \approx 17.0391 \qquad \text{and} \qquad p(5) \approx 40.5431.
$$
Therefore $F_{\la \la}(\la, t) \leq F_{\la \la}(0.82, t)$ and
$$
F_{\la \la}(0.82, t) < t^2- 4.92\left(1+\frac{1}{6560}\right)\left(t+1-\frac{\left(t-\frac{3\pi}{2}\right)^2}{2}\right) + 0.016.
$$
The right-hand side is a convex parabola $q(t)$ and therefore to prove it is negative it suffices to evaluate it at the endpoints $t = 2.8$ and $t = 5$. We have
\[
q(2.8) \approx -1.8447 \qquad \text{and} \qquad q(5) \approx -4.305.\qedhere
\]
\end{proof}

\begin{proof}[Proof of Lemma~\ref{lem:0.5}]
We have the identity
\begin{align*}
\la^{-1} F_{tt}(\la, t) = 1- &\frac{\left(t- \frac{3\pi}{2} \right)^2}{2} + \la\left( 1- \frac{\left(\la t- \frac{3\pi}{2} \right)^2}{2}\right)\\
&-(1-\la^2)\left(1+\frac{1}{6560}\right) \left(1 + \frac{t}{6561}\right).
\end{align*}
Hence, neglecting the terms involving $1/6560$ and $t/6561$,
$$
2 F_{tt}(0.5, t) \leq \frac{3}{4} - \frac{\left(t- \frac{3\pi}{2} \right)^2}{2} - \frac{\left(t- 3\pi \right)^2}{16}. 
$$
The right-hand side is a parabola which always lies below zero. Therefore $F(0.5, t)$ is a concave function and we only need to check that both $F(0.5, 2.8)$ and $F(0.5, 5)$ are positive. Indeed,
\[
F(0.5, 2.8) \approx 0.3448 \qquad \text{and} \qquad F(0.5, 5) \approx 2.2033.\qedhere 
\]
\end{proof}

\begin{proof}[Proof of Lemma~\ref{lem:0.8194}]
This part is the most problematic because the minimum of $F(0.8194, t)$ in the specified interval is roughly $0.0002$. We proceed in the following way. First we separate the positive from the negative terms: $F(0.8194, t) = \varphi(t) - \psi(t)$ where
\begin{align*}
\varphi(t) &= 0.8194 \frac{\left(t - \frac{3\pi}{2} \right)^2}{2} + \frac{\left(0.8194t - \frac{3\pi}{2} \right)^2}{2} + d(t+1), \\
\psi(t) &= 0.8194\left(1 + \frac{\left(t - \frac{3\pi}{2} \right)^4}{4!}\right) + 1 + \frac{\left(0.8194t - \frac{3\pi}{2} \right)^4}{4!} + d\left(\frac{t^3}{39366} + \frac{\left(t - \frac{3\pi}{2} \right)^2}{2} \right) 
\end{align*}
and $d = 0.8194 \cdot (1-0.8194^2) \cdot (1+1/6560)$. We are going to split the interval $[2.8, 5.78]$ in smaller overlapping subintervals, and for each of these subintervals we will prove that for some line $L(t)$ we have $\varphi(t) \geq L(t) \geq \psi(t)$. Since $\psi''(t) \geq 0$ the second inequality needs only to be checked at the endpoints of the subinterval. To prove the first one we will show that the discriminant of the quadratic polynomial $\varphi(t) - L(t)$ is negative, and therefore these two curves never intersect. The six lines we consider are the following:
\begin{alignat*}{2}
&L_1(t) = -1.5t + 8.5, \qquad &&L_2(t) = -0.45t + 4.02, \\
&L_3(t) = 0.025t + 1.71372, \qquad &&L_4(t) = 0.077t + 1.4519, \\
&L_5(t) = 0.2t + 0.8256, \qquad &&L_6(t) = t - 3.5.
\end{alignat*}

Writing $\varphi(t) = a t^2 + bt + c$ and expanding,
$$
\varphi(t) = 0.74540818t^2 + (d-0.8194 \cdot 3\pi)t + (d + 1.8194 \cdot 9\pi^2/8),
$$
\emph{i.e.}, $a = 0.74540818$, $b \approx -7.45338058$ and $c \approx 20.47063551$, the error being smaller than $10^{-8}$. For each of the six lines considered $L_i(t) = B_i t + C_i$ we compute now the discriminant $\Delta_i = (b-B_i)^2 - 4a(c-C_i)$ of $\varphi-L_i$, within an error smaller than $10^{-6}$:
\begin{alignat*}{3}
&\Delta_1 \approx -0.249298,\qquad &&\Delta_2 \approx -0.002414,\qquad &&\Delta_3 \approx -0.000057,\\
&\Delta_4 \approx -0.000252,\qquad &&\Delta_5 \approx -0.000046,\qquad &&\Delta_6 \approx -0.011988.
\end{alignat*}

We consider as our partition of $[2.8, 5.78]$ the intervals $[t_i, t_{i+1}]$ where $t_1 = 2.8$, $t_7 = 5.78$ and $t_i$ for $2 \leq i \leq 6$ is the abscissa of the intersection between the lines $L_{i-1}$ and $L_i$. They have the following exact values:
\begin{alignat*}{3}
&t_1 = 2.8, \qquad &&t_2 = 64/15\approx 4.27, \qquad &&t_3 = 57657/11875 \approx 4.86,\\
&t_4 = 5.035,\qquad &&t_5 = 6263/1230 \approx 5.09, \qquad &&t_6 = 5.407, \qquad t_7 = 5.78.
\end{alignat*}
Note that we only need to check $L_i(t_i) \geq \psi(t_i)$ for $1 \leq i \leq 6$ and $L_6(t_7) \geq \psi(t_7)$. This can be verified from the following table, where we have included $L_7 = L_6$. The values are computed within an error of size $10^{-7}$.

\begin{center}
\begin{tabular}{r|r|r}
$i$  & $L_i(t_i)$ & $\psi(t_i)$ \\\hline
1    &           4.3 & 4.1931243 \\
2    &           2.1 & 1.9392134 \\
3    & 1.8351032 & 1.8350379 \\
4    & 1.839595 & 1.8395934 \\
5    & 1.843974 & 1.843946 \\
6    &      1.907 & 1.8936546 \\
7    &        2.28 & 2.0185385
\end{tabular}
\end{center}
We conclude $F(0.8194, t) \geq 0$ for $2.8 \leq t \leq 5.78$ and the proof is finished.
\end{proof}

\section*{Acknowledgments}
During the elaboration of this article the first author has been partially supported by grant MTM2015-65792-P by MINECO/FEDER-EU and by the Thematic Research Network MTM2015-69323-REDT, MINECO, Spain. The second author has been supported by the international PhD program ``la Caixa''-Severo Ochoa at the Instituto de Ciencias Matem\'aticas (CSIC-UAM-UC3M-UCM).

%%%%%%%%%%%%%%%%%%%%%%%%%%%%%%%%%%
\end{document}